\documentclass[fleqn,3p,12pt,times]{elsarticle}
\usepackage[utf8]{inputenc}
\usepackage{amssymb,latexsym}
\usepackage{amsfonts}
\usepackage{amsmath}
\usepackage{enumerate}
\usepackage{bbm}
\usepackage{xcolor}
\usepackage{booktabs}
\usepackage{float}
\usepackage{amssymb}
\usepackage{amsthm}
\usepackage{comment}
\usepackage{url}
\usepackage{enumitem}
\usepackage{natbib}
\usepackage{bm}
\usepackage{amsmath}

\makeatletter
\newcommand{\eqnum}{\leavevmode\hfill\refstepcounter{equation}\textup{\tagform@{\theequation}}}
\makeatother

\newtheorem{ex}{Example}
\newtheorem{thm}{Theorem}

\newcommand\dist{\buildrel d \over =}

\journal{arXiv}

\begin{document}
\begin{frontmatter}

\title{The limit law of certain discrete multivariate distributions}



\author{Andrius Grigutis\fnref{myfootnote1}}
\author{Artur Nakliuda\fnref{myfootnote2}}
\address{Institute of Mathematics, Vilnius University, Naugarduko 24, LT-03225 Vilnius}

\fntext[myfootnote1]{Email: andrius.grigutis@mif.vu.lt}
\fntext[myfootnote2]{Email: artur.nakliuda@mif.stud.vu.lt}

\begin{abstract}
Let $X_1,\,X_2,\,\ldots,\,X_N$, $N\in\mathbb{N}$ be independent but not necessarily identically distributed discrete and integer-valued random variables. Assume that $X_1\geqslant m_1$, $X_2\geqslant m_2$, $\ldots$, $X_N\geqslant m_N$ almost surely, where $m_1,\,m_2,\ldots,\,m_N$ are some integer numbers such that $m_1+m_2+\ldots+m_N<0$, and $X_k\dist X_{k+N}$, for all $k\in\mathbb{N}$ in the sequence $X_1,\,X_2,\,\ldots$ In this communication, we make use of some of the known results to provide the closed-form expression of the limit multivariate distribution function $\mathbb{P}(X_1\leqslant x,\,X_1+X_2\leqslant x,\,\ldots)$, $x\in\mathbb{Z}$ via: (1) inclusion-exclusion principle based product of the roots of $G_{N}(s)=1$, where $G_{N}(s)$ is the probability generating function of $S_N=X_1+X_2+\ldots+X_N$, (2) the probability mass function of $S_N$, and (3) the expectation $\mathbb{E}S_N$.
\end{abstract}

\end{frontmatter}

This work consists of communication on the setup of the distribution function
\begin{align}\label{dist_f}
F_{\infty}(x):=\lim_{n\to\infty}\mathbb{P}(X_1\leqslant x,\,X_1+X_2\leqslant x,\,\ldots,\,X_1+X_2+\ldots+X_n\leqslant x),\,x\in\mathbb{Z}.
\end{align}

The main ideas in the setup of $\eqref{dist_f}$ stem from \cite{AIMS_Math} and \cite{Aims_PUQR}, where similar problems are considered in a rather applied context. Moreover, in \cite{AIMS_Math} and \cite{Aims_PUQR} there are more restrictive conditions (i.i.d.) assumed on the random variables $X_1,\,X_2,\,\ldots$, and the methods are more extensive compared to the ones presented here.

Here we fix some $N\in\mathbb{N}$ and assume that:\\
(a) $X_1,\,X_2,\,\ldots,\,X_N$ are discrete, integer-valued, and independent random variables that may be distributed differently.\\
(b) $\mathbb{P}\left(X_1\geqslant m_1,\,X_2\geqslant m_2,\,\ldots,\,X_N\geqslant m_N\right)=1$ and $\mathbb{P}\left(X_1=m_1,\,X_2=m_2,\,\ldots,\,X_N=m_N\right)$ $>0$, where  $m_1,\,m_2,\ldots,\,m_N$ are integers such that $-D:=m_1+m_2+\ldots+m_N<0$.\\
(c) $X_k\dist X_{k+N}$, for all $k\in\mathbb{N}$ in the sequence $X_1,\,X_2,\,\ldots$ (of course, $N=1$ refers to the i.i.d. case).

The definition of $F_{\infty}(x)$ and the above assumptions (a)--(c) imply that there is enough to consider $F_{\infty}(x)$ for $x\in\mathbb{Z}$ such that 
$$
x\geqslant M:=\max\{m_1,\,m_1+m_2,\,\ldots,\,m_1+m_2+\ldots+m_N\}\geqslant-D=m_1+m_2+\ldots+m_N
$$
because $F_{\infty}(x)=0$ always if $x<M$. We will see later on that $F_{\infty}(x)\not\equiv0$ if $\mathbb{E}(X_1+X_2+\ldots+X_N)<0$, except in some cases when all of the random variables $X_1,\,X_2,\,\ldots,\,X_N$ are degenerate. Let $S_N:=X_1+X_2+\ldots+X_N$ and denote:
\begin{align*}
f_N(x)&:=\mathbb{P}(S_N=x),\,x\geqslant-D,\\
F_N(x)&:=\mathbb{P}(S_N\leqslant x),\,x\geqslant-D.
\end{align*}
Then, the elementary rearrangements imply
\begin{align*}
F_{\infty}(x)&=\mathbb{P}\left(\bigcap_{n=1}^{\infty}\left\{\sum_{j=1}^{n}X_j\leqslant x\right\}\right)
=\mathbb{P}\left(\bigcap_{n=1}^{N}\left\{\sum_{j=1}^{n}X_j\leqslant x\right\} 
\cap
\bigcap_{n=N+1}^{\infty}\left\{\sum_{j=1}^{n}X_j\leqslant x\right\}\right)\\
&=\mathbb{P}\left(\bigcap_{n=1}^{N}\left\{\sum_{j=1}^nX_j\leqslant x\right\}\cap
\bigcap_{n=N+1}^{\infty}\left\{\sum_{j=N+1}^n X_j\leqslant x-\sum_{j=1}^N X_j\right\}\right)\\
&=\sum_{\substack{m_1\leqslant i_1\leqslant x\\m_2\leqslant i_1+i_2\leqslant x\\
\vdots \vspace{1mm} \\
m_N\leqslant i_1+i_2+\ldots+i_N\leqslant x}}\hspace{-5mm}\mathbb{P}(X_1=i_1)\mathbb{P}(X_2=i_2)\cdots\mathbb{P}(X_N=i_N)\,
F_\infty\left(x-\sum_{j=1}^Ni_j\right)
\end{align*}
and that yields 
\begin{align}
F_{\infty}(x)&=\sum_{j=-D}^{x}f_N(j)\,F_{\infty}(x-j)=\sum_{j=0}^{x+D}f_N(x-j)\,F_{\infty}(j),\,x\geqslant M,\,M\leqslant0,\label{eq:recurrence_M<0}\\
F_{\infty}(x)&=\sum_{j=-D}^{x-M}f_N(j)\,F_{\infty}(x-j)=\sum_{j=M}^{x+D}f_N(x-j)\,F_{\infty}(j),\,x\geqslant M>0.\label{eq:recurrence_M>0}
\end{align}

The obtained relations \eqref{eq:recurrence_M<0} and \eqref{eq:recurrence_M>0} may be recognized as discrete truncated versions of the Wiener-Hopf equation whose kernel is a probability density, see \cite[eq. (1.1)]{Spitzer_Duke}. 

The recurrence \eqref{eq:recurrence_M<0} implies that it is sufficient to know $F_{\infty}(0),\,F_{\infty}(1)$, $\ldots$, $F_{\infty}(D-1)$, since any other desired value of $F_{\infty}(x)$, $x\geqslant M$ can be computed using the same recurrence \eqref{eq:recurrence_M<0}. The same logic applies to the recurrence \eqref{eq:recurrence_M>0}, where it is sufficient to know
$F_{\infty}(M),\,F_{\infty}(M+1),\,\ldots,\,F_{\infty}(M+D-1)$. We now explain how to get the required initial values for the recurrence \eqref{eq:recurrence_M<0} and notice that the procedure is nearly identical when finding the initial values for \eqref{eq:recurrence_M>0}.

    Let $s\in\mathbb{C}$ and denote two generating functions 
\begin{align*}
G_N(s)&=\sum_{j=-D}^{\infty}s^jf_N(j),\,0<|s|\leqslant1,\\
\Xi(s)&=\sum_{j=0}^{\infty}s^{j}F_{\infty}(j),\,|s|<1.
\end{align*}
Then, recurrence \eqref{eq:recurrence_M<0} implies
$$
F_{\infty}(x-D)=\sum_{j=0}^{u}f_N(x-D-j)\,F_{\infty}(j),\,x\geqslant M+D,
$$
which yields the relation between the defined generating functions (\cite[eq. (3)]{Landriault}, \cite[eq. (13)]{GJ})
\begin{align*}
\Xi(s)&=\sum_{x=D}^{\infty}s^{x-D}F_{\infty}(x-D)=\sum_{x=D}^{\infty}s^{x-D}\left(\sum_{j=0}^{x}f_N(x-D-j)\,F_{\infty}(j)\right)\\
&=\sum_{x=0}^{\infty}\left(\sum_{j=0}^{x}f_N(x-D-j)\,F_{\infty}(j)\right)s^{x-D}
-\sum_{u=0}^{D-1}\left(\sum_{j=M}^{x}f_N(x-D-j)\,F_{\infty}(j)\right)s^{x-D}\\
&=\Xi(s)G_N(s)-\sum_{j=0}^{D-1}F_{\infty}(j)\sum_{x=j}^{D-1}f_N(x-D-j)\,s^{x-D},\,0<|s|<1,
\end{align*}
and that is
\begin{align}\label{eq:P-1}
\Xi(s)\left(s^D G_N(s)-s^D\right)=\sum_{j=0}^{D-1}F_{\infty}(j)\sum_{x=j}^{D-1}f_N(x-D-j)\,s^{x}, |s|<1.
\end{align}

We aim to evaluate \eqref{eq:P-1} by taking $s$ derivative of both sides and letting $s\to1$ afterward. 

    Let us denote the non-negative and integer-valued random variable 
$\xi$, whose probability mass is
\begin{align*}
f_{\infty}(0)&=\mathbb{P}(\xi=0)=F_{\infty}(0),\\
f_{\infty}(k)&=\mathbb{P}(\xi=k)=F_{\infty}(k)-F_{\infty}(k-1),\,k\in\mathbb{N}.
\end{align*}
It can be shown that $\mathbb{P}(\xi<\infty)=1$ if $\mathbb{E}S_N<0$, see \cite[Lem. 1]{GJS}.
Let $G_{\xi}(s)$ denote the $\xi$ probability-generating function
\begin{align*}
G_{\xi}(s)=\sum_{j=0}^{\infty}s^{j}f_{\infty}(j),\,|s|\leqslant1.
\end{align*}
Then, the generating functions $G_{\xi}(s)$ and $\Xi(s)$ are related as follows
\begin{align*}
\Xi(s)=\sum_{j=0}^{\infty}s^{j}F_{\infty}(j)=\sum_{j=0}^{\infty}s^{j}\sum_{k=0}^{j}f_{\infty}(k)=\sum_{k=0}^{\infty}f_{\infty}(k)\sum_{j=k}^{\infty}s^{j}=\frac{G_{\xi}(s)}{1-s},\,|s|<1.
\end{align*}
Thus, according to \eqref{eq:P-1}, we get
\begin{align}\label{eq:P-1_v2}
G_{\xi}(s)\left(s^D G_N(s)-s^D\right)=(1-s)\sum_{j=0}^{D-1}F_{\infty}(j)\sum_{x=j}^{D-1}f_N(x-D-j)\,s^{x},\,|s|\leqslant1.
\end{align}
Then, by taking $s$ derivative of both sides of the last equality (under $\mathbb{E}S_N<0$) and letting $s\to1$ afterward, we obtain
\begin{align}\label{main_eq_last}
\sum_{j=0}^{D-1}F_{\infty}(j)\,F_{N}(-j-1)=
-\mathbb{E}S_N,
\end{align}
see \cite[p. 5191]{AIMS_Math}, \cite[Lem. 5]{GJS}.

Let $\alpha_1,\,\alpha_2,\,\ldots,\,\alpha_{D-1}$ be the roots of $G_N(s)=1$ in $|s|\leqslant1,\,s\neq1$. There are exactly $D-1$ such roots counted with their multiplicities in the provided circle, see \cite[Lem. 3.3]{Aims_PUQR}, \cite[Prop. 2]{Landriault}, \cite[p. 719]{Monthly}, \cite[Sec. 4]{GJ}, and the left-hand-side of \eqref{eq:P-1_v2} vanishes if s is a root of $G_N(s)=1$. 

Let us assume that the roots $\alpha_1,\,\alpha_2,\,\ldots,\,\alpha_{D-1}$ of $G_N(s)=1$ are simple. By replicating the equation \eqref{eq:P-1_v2} over these roots and including \eqref{main_eq_last}, we obtain the following system of linear equations
\begin{align}\label{syst:main}
&\begin{pmatrix}
\sum\limits_{x=0}^{D-1}f_N(x-D)\alpha_1^x
&\sum\limits_{x=1}^{D-1}f_N(x-D-1)\alpha_1^x
&\ldots
&f_N(-D)\,\alpha_1^{D-1}\\
\sum\limits_{x=0}^{D-1}f_N(x-D)\alpha_2^x
&\sum\limits_{x=1}^{D-1}f_N(x-D-1)\alpha_2^x
&\ldots
&f_N(-D)\,\alpha_2^{D-1}\\
\vdots&\vdots&\ddots&\vdots\\
\sum\limits_{x=0}^{D-1}f_N(x-D)\alpha_{D-1}^x
&\sum\limits_{x=1}^{D-1}f_N(x-D-1)\alpha_{D-1}^x
&\ldots
&f_N(-D)\,\alpha_{D-1}^{D-1}\\
F_{N}(-1)&F_{N}(-2)&\ldots&f_{N}(-D)
\end{pmatrix}
\begin{pmatrix}
F_{\infty}(0)\\
F_{\infty}(1)\\
\vdots\\
F_{\infty}(D-1)
\end{pmatrix}\nonumber\\
&=
\begin{pmatrix}
0,\,
0,\,
\ldots,\,
0,\,
-\mathbb{E}S_N
\end{pmatrix}^T,
\end{align}
whose solution is:
\begin{align}
&F_{\infty}(0)=\frac{-\mathbb{E}S_N}{f_N(-D)}\prod_{j=1}^{D-1}\frac{\alpha_j}{\alpha_j-1},\label{sol_0}\\
&F_{\infty}(1)=-\frac{F_{N}(-D+1)}{f_N(-D)}F_{\infty}(0)\label{sol_1}\\
&-\frac{\mathbb{E}S_N}{f_N(-D)}\prod_{j=1}^{D-1}\frac{1}{\alpha_j-1}\left(\prod_{j=1}^{D-1}\alpha_j-\sum\limits_{1\leqslant j_1<\ldots<j_{D-2}\leqslant D-1}
\alpha_{j_1}\cdots\alpha_{j_{D-2}}\right),\nonumber\\
&F_{\infty}(2)=-\frac{F_{N}(-D+1)}{f_N(-D)}F_{\infty}(1)-\frac{F_{N}(-D+2)}{f_N(-D)}F_{\infty}(0)-\frac{\mathbb{E}S_N}{f_N(-D)}\prod_{j=1}^{D-1}\frac{1}{\alpha_j-1}\label{sol_2}\\
&\times\left(\prod_{j=1}^{D-1}\alpha_j-\sum_{1\leqslant j_1<\ldots<j_{D-2}\leqslant D-1}
\alpha_{j_1}\cdots\alpha_{j_{D-2}}
+\sum_{1\leqslant j_1<\ldots<j_{D-3}\leqslant D-1}
\alpha_{j_1}\cdots\alpha_{j_{D-3}}\right),\nonumber\\
&\,\,\vdots\nonumber\\
&F_{\infty}(D-1)=-\frac{1}{f_N(-D)}\sum_{i=1}^{D-1}F_{N}(-i)F_{\infty}(i)
-\frac{\mathbb{E}S_N}{f_N(-D)}\prod_{j=1}^{D-1}\frac{1}{\alpha_j-1}
\times\Bigg(\prod_{j=1}^{D-1}\alpha_j-\label{sol_D-1}\\
&\sum_{1\leqslant j_1<\ldots<j_{D-2}\leqslant D-1}
\alpha_{j_1}\cdots\alpha_{j_{D-2}}
+\sum_{1\leqslant j_1<\ldots<j_{D-3}\leqslant D-1}
\alpha_{j_1}\cdots\alpha_{j_{D-3}}+\ldots+(-1)^{D+1}\Bigg),\nonumber
\end{align}
see \cite[Thm. 3.3]{AIMS_Math}, \cite[Thm. 2.5]{Aims_PUQR}.

Let us return briefly to the initial values $F_{\infty}(M),\,F_{\infty}(M+1),\,\ldots,\,F_{\infty}(M+D-1)$ for the recurrence \eqref{eq:recurrence_M>0} when $M>0$. Following the above-listed argumentation, we shall define 
$$
\tilde{\Xi}(s)=\sum_{j=M}^{\infty}s^{j-M}F_{\infty}(j),\,|s|<1.
$$
Then, following the derivation of \eqref{eq:P-1}, we shall derive
\begin{align}\label{eq:gen_relation_v2}
\tilde{\Xi}(s)s^D\left(G_N(s)-1\right)=\sum_{j=M}^{M+D-1}F_{\infty}(j)\sum_{x=j}^{M+D-1}f_N(x-D-j)s^{x-M},\,|s|<1,
\end{align}
which, under the above-listed argumentation, yields an identical system to \eqref{syst:main} and its solution \eqref{sol_0}-\eqref{sol_D-1}, where $F_{\infty}(0),\,F_{\infty}(1),\,\ldots,\,F_{\infty}(D-1)$ gets replaced with $F_{\infty}(M)$, $F_{\infty}(M+1),\,\ldots,\,F_{\infty}(M+D-1)$ respectively. Note that if there are multiple roots among those $\alpha_1,\,\alpha_2,\,\ldots,\,\alpha_{D-1}$ we shall rebuild the system \eqref{syst:main} replacing its coefficients with derivatives, see \cite[Lem. 4.3]{AIMS_Math}.

Let us comment on the assumption of finite left supports of the random variables 
$X_1,\,X_2,\,\ldots,\,X_N$, i.e. $\mathbb{P}(X_1\geqslant m_1,\,X_2\geqslant m_2,\,\ldots,\,X_N\geqslant m_N)=1$. This assumption for the presented methods is ne\-cessary because if at least one $m_1,\,m_2,\,\ldots,\,m_N$ is $"-\infty"$, the sum in \eqref{eq:recurrence_M<0} or \eqref{eq:recurrence_M>0} runs to infinity too. Then the computation of $F_{\infty}(x)$ leads to the usage of the Pollaczek–Khinchine formula, see \cite[eq. (10)]{EMBRECHTS198255}. The Pollaczek–Khinchine formula, in turn, is based on the direct computation of the distribution functions $\mathbb{P}(X_1\leqslant x),\,\mathbb{P}(X_1+X_2\leqslant x),\,\ldots$, and the success of its usage becomes heavily dependent on the neatness of the distributions $X_1,\,X_2,\,\ldots,\,X_N$; see \cite{Willmot2017}.

    We summarize the listed thoughts on the distribution function 
$F_{\infty}(x),\,x\in\mathbb{Z}$ in the following theorem.
\begin{thm}\label{thm}
Let $X_1,\,X_2,\,\ldots,\,X_N$, $N\in\mathbb{N}$ be discrete, integer-valued, and independent random variables. Say that
$\mathbb{P}\left(X_1\geqslant m_1,\,X_2\geqslant m_2,\,\ldots,\,X_N\geqslant m_N\right)=1$ and \\$\mathbb{P}\left(X_1=m_1,\,X_2=m_2,\,\ldots,\,X_N=m_N\right)>0$, where  $m_1,\,m_2,\ldots,\,m_N$ are integers such that $m_1+m_2+\ldots+m_N=-D<0$.
Moreover, let $X_k\dist X_{k+N}$, for all $k\in\mathbb{N}$ in the sequence $X_1,\,X_2,\,\ldots$, and recall that:
\begin{align*}
M&=\max\{m_1,\,m_1+m_2,\,\ldots,\,m_1+m_2+\ldots+m_N\},\\ S_N&=X_1+\ldots+X_N,\\ 
f_N(k)&=\mathbb{P}(S_N=k),\,k\in\mathbb{Z},\\
G_N(s)&=s^{-D}f_N(-D)+s^{-D+1}f_N(-D+1)+\ldots,\,0<|s|\leqslant1,\,s\in\mathbb{C},\\ F_{\infty}(x)&=\mathbb{P}(X_1\leqslant x,\,X_1+X_2\leqslant x,\,\ldots),\, x\in\mathbb{Z}.
\end{align*}
Then:
\begin{align*}
&(\pmb{i})\,F_{\infty}(x)\equiv0,\textit{ if } \mathbb{E}S_N>0 \textit{ or } \mathbb{E}S_N=0 \textit{ with } \mathbb{P}(S_N=0)<1.\\
&(\pmb{ii})\, F_{\infty}(x)=0 \textit{ when } x<M \textit{ and } F_{\infty}(x)=1 \textit{ when } x\geqslant M, \textit{ if } \mathbb{E}S_N=0 \textit{ with } \mathbb{P}(S_N=0)=1.\\
&(\pmb{iii})\,F_{\infty}(x)=0 \textit{ for all } x<M; F_{\infty}(0),\,F_{\infty}(1),\,\ldots,\,F_{\infty}(D-1) \textit{ are given in } \eqref{sol_0}-\eqref{sol_D-1},\\
&\textit{ and any other value } F_{\infty}(u),\,u\geqslant M \textit{ is computed by } \eqref{eq:recurrence_M<0} \textit{ when } M\leqslant0,\, \mathbb{E}S_N<0 \textit{ and }\\
&\alpha_1,\,\alpha_2,\,\ldots\alpha_{D-1},\text{ are the simple roots of } G_N(s)=1 \textit{ in } |s|\leqslant1,\,s\neq1.\\
&(\pmb{iv}) \,F_{\infty}(x)=0 \textit{ for all } x<M; F_{\infty}(M),\,F_{\infty}(M+1),\,\ldots,\,F_{\infty}(M+D-1) \textit{ are computed by }\\
&\eqref{sol_0}-\eqref{sol_D-1} \textit{ when they replace } F_{\infty}(0),\,F_{\infty}(1),\,\ldots,\,F_{\infty}(D-1) \textit{ there respectively},\\
&\textit{ and any other value } F_{\infty}(x),\,x\geqslant M \textit{ is computed by } \eqref{eq:recurrence_M>0} \textit{ when } M>0,\, \mathbb{E}S_N<0 \textit{ and }\\
&\alpha_1,\,\alpha_2,\,\ldots\alpha_{D-1},\text{ are the simple roots of } G_N(s)=1 \textit{ in } |s|\leqslant1,\,s\neq1.\\
&(\pmb{v})\textit{ The generating function } \Xi(s)=F_{\infty}(0)+F_{\infty}(1)s+\ldots \text{ for } s \text{ such that } |s|<1,\\
&G_{N}(s)\neq1 \textit{ is given by } \eqref{eq:P-1} \textit{ when }
F_{\infty}(0),\,F_{\infty}(1),\,\ldots,\,F_{\infty}(D-1) \text{ are known, } M\leqslant0, \textit{ and }\\ 
&\mathbb{E}S_N<0.\\
&(\pmb{vi})\textit{ The generating function } \tilde{\Xi}(s)=F_{\infty}(M)+F_{\infty}(M)s+\ldots \text{ for } s \text{ such that } |s|<1,\\
&G_{N}(s)\neq1 \textit{ is given by } \eqref{eq:gen_relation_v2} \textit{ when }
F_{\infty}(M),\,F_{\infty}(M+1),\,\ldots,\,F_{\infty}(M+D-1) \text{ are known, }\\
&M>0, \textit{ and } \mathbb{E}S_N<0.
\end{align*}
\end{thm}
\begin{proof}
(\textit{\textbf{i}}) If $\mathbb{E}S_N>0$, then $F_{\infty}(x)\equiv0$ is implied by the strong law of large numbers, see \cite[Thm. 2]{GJS}, \cite[Prop. 7.23]{Resnic}. If $\mathbb{E}S_N=0$ and $\mathbb{P}(S_N=0)<1$, then $F_{\infty}(x)\equiv0$ can be derived by the means of \cite{Grigutis_Karbonskis_Šiaulys_2023}. 

(\textit{\textbf{ii}}) The provided conditions imply that $X_1=m_1$, $X_2=m_2$, $\ldots$, $X_N=m_N$ with probability one and $m_1+m_2+\ldots+m_N=0$. Thus, it is sufficient to establish the minimal threshold $x$ such that the conditions $X_1\leqslant x$, $X_1+X_2\leqslant x$, $\ldots$, $X_1+X_2+\ldots+X_N\leqslant x$ are all satisfied.

(\textit{\textbf{iii}})--(\textit{\textbf{vi}}) are implied by the means given before Theorem \ref{thm}.
\end{proof}
In this work, we provide two examples of the limit multivariate distribution function \eqref{dist_f} construction. All the necessary computations are performed using the software \cite{Mathematica}. 
\begin{ex}
Let $\mathbb{P}(X=-3)=\mathbb{P}(X=1)=1/2$ and say $X_1,\,X_2,\,\ldots$ are independent copies of $X$. Then
$$
\alpha_1=\frac{2-\left(1-i\sqrt{3}\right)\sqrt[3]{19-3\sqrt{33}}-\left(1+i\sqrt{3}\right)\sqrt[3]{19+3\sqrt{33}}}{6}\approx-0.419643-0.606291i
$$
and its conjugate $\alpha_2=\overline{\alpha}_1$ are the simple roots of
$$
\frac{1}{2s^3}+\frac{s}{2}=1.
$$
Then, according to Theorem \ref{thm},
\begin{align*}
F_{\infty}(x)&=0,\,x<-3,\\
F_{\infty}(-3)&=\frac{\alpha_1\alpha_2}{(\alpha_1-1)(\alpha_2-1)}\approx0.228155,\\
F_{\infty}(-2)&=\frac{-(\alpha_1+\alpha_2)}{(\alpha_1-1)(\alpha_2-1)}\approx0.352201,\\
F_{\infty}(-1)&=\frac{1}{(\alpha_1-1)(\alpha_2-1)}\approx 0.419643,\\
F_{\infty}(0)&=\frac{2\alpha_1\alpha_2}{(\alpha_1-1)(\alpha_2-1)}\approx 0.456311,\\
F_{\infty}(1)&=\frac{-2(\alpha_1+\alpha_2)}{(\alpha_1-1)(\alpha_2-1)}\approx 0.704402,\\
F_{\infty}(2)&=\frac{2}{(\alpha_1-1)(\alpha_2-1)}\approx0.839287,
\end{align*}
and we may proceed using the recurrence \eqref{eq:recurrence_M<0}
\begin{align*}
F_{\infty}(x)=2\left(F_{\infty}(x-3)-\sum_{j=0}^{x-1}f_1(x-3-j)F_{\infty}(j)\right),\,x\geqslant0,
\end{align*}
where $\sum_{j=0}^{-1}:=0$.
    In this example, the random values of the partial sums $\{X_1,\,X_1+X_2,\,\ldots\}$ 
can be illustrated as follows 
\begin{align*}
&\hspace{6cm}-3 \quad {\color{red} 1} \\
&\hspace{5.5cm}-6 \quad -2 \quad {\color{red} 2} \\ 
&\hspace{5cm}-9\quad -5 \quad -1 \quad {\color{red} 3} \\ 
&\hspace{4.5cm}-12\quad -8\quad -4\quad {\color{red}0}\quad {\color{red} 4} \\
&\hspace{4cm}-15\quad -11 \quad -7 \quad -3 \quad {\color{red} 1}\quad {\color{red} 5}\\
&\hspace{3.5cm}-18\quad -14 \quad -10 \quad -6 \quad -2\quad {\color{red} 2}\quad {\color{red} 6}\\
&\hspace{3cm}-21\quad -17 \quad -13 \quad -9 \quad -4\quad -1\quad {\color{red} 3}\quad {\color{red} 7}\\
&\hspace{2.5cm}-24\quad -20 \quad -16 \quad -12 \quad -8\quad -4\quad {\color{red} 0}\quad {\color{red} 4}\quad {\color{red} 8}\\
&\hspace{7cm} \ldots
\end{align*}
where the distributed underlying probabilities follow the binomial law $(1/2+1/2)^n,\,n\in\mathbb{N}$. For instance, $F_{\infty}(-1)\approx 0.419643$ is the probability that we never pick up the red number according to the decision of the coin flip at every step when moving downwards (ultimately) from the top in the provided triangle. 
\end{ex}

\begin{ex}
Let
\begin{align*}
\mathbb{P}(X=k)&=0.55(0.45)^{k-1},\,k=1,\,2,\,\ldots,\\
\mathbb{P}(Y=k)&=e^{-1/2}\frac{(1/2)^{k+3}}{(k+3)!},\,k=-3,\,-2,\,\ldots,\\
\mathbb{P}(Z=k)&=e^{-k}-e^{-k-1},\,k=0,\,1,\,\ldots
\end{align*}
Here the random variable $X$ can be recognized as geometric with parameter $p=0.55$, $Y$ as shifted Poisson with parameter $\lambda=1/2$, and $Z$ as discrete Weibull with unit parameters. We suppose that $X_1,\,X_4,\,X_7,\,\ldots$ are independent copies of $X$; $X_2,\,X_5,\,X_8,\,\ldots$ are independent copies of $Y$;
$X_3,\,X_6,\,X_{9},\,\ldots$ are independent copies of $Z$. Then, $\mathbb{E}S_3=\mathbb{E}(X+Y+Z)\approx-0.10178<0$ and
$$
\frac{0.55}{1-0.45s}\cdot\frac{e^{-1/2+s/2}}{s^2}\cdot\frac{e-1}{e-s}=1
$$
has a simple root $s\approx-0.364796=:\alpha$. The distribution function of interest is:
\begin{align*}
F_{\infty}(x)&=0,\,x<1,\\
F_{\infty}(1)&=\frac{\mathbb{E}S_3}{f_3(-2)}\cdot\frac{\alpha}{\alpha-1}\approx 0.129012,\\
F_{\infty}(2)&=-\frac{F_3(-1)}{f_3(-2)}\cdot F_{\infty}(1)-\frac{\mathbb{E}S_3}{f_3(-2)}\approx0.183634,
\end{align*}
and we may proceed using the recurrence \eqref{eq:recurrence_M>0}
\begin{align*}
F_{\infty}(x)=\frac{1}{f_3(-2)}\left(F_{\infty}(x-2)-\sum_{j=1}^{x-1}f_3(x-2-j)F_{\infty}(j)\right),\,x\geqslant 3.
\end{align*}
\end{ex}

See \cite[Sec. 6]{AIMS_Math} and \cite[Sec. 5]{Aims_PUQR} for more related examples.

\bibliography{mybibfile} 

\end{document}